%% file: KL_DN_final.tex
\documentclass[a4paper,11pt]{amsart}
\usepackage{amssymb,amsmath,amsthm}
\usepackage[normalem]{ulem}
\usepackage{fullpage}
\usepackage{color}
\definecolor{darkblue}{rgb}{0.2,0.2,0.6}
\definecolor{darkred}{rgb}{0.6,0.1,0.1}
\usepackage[colorlinks=true,linkcolor=darkblue, citecolor=darkred]{hyperref}
\usepackage{eucal, beton, lmodern}
\usepackage{stackrel, enumerate, eurosym}
\usepackage{graphicx, bbm, tcolorbox}
\DeclareMathOperator*{\curl}{curl}

\newcommand{\eg}{\emph{e.g.}}

\input commands.tex

\usepackage[left=2.7cm, right=2.7cm, marginparwidth=2cm, textheight = 24cm ]{geometry}
\usepackage[normalem]{ulem}
\definecolor{DarkGreen}{rgb}{0,0.5,0.1}

\definecolor{DarkBlue}{rgb}{0,0.1,0.5}

\newcommand\soutD{\bgroup\markoverwith
	{\textcolor{DarkGreen}{\rule[.5ex]{2pt}{1pt}}}\ULon}
\newcommand\soutP{\bgroup\markoverwith
	{\textcolor{blue}{\rule[.5ex]{2pt}{1pt}}}\ULon}
\newcommand{\Hm}[1]{\leavevmode{\marginpar{\tiny%
			$\hbox to 0mm{\hspace*{-0.5mm}$\leftarrow$\hss}%
			\vcenter{\vrule depth 0.1mm height 0.1mm width \the\marginparwidth}%
			\hbox to
			0mm{\hss$\rightarrow$\hspace*{-0.5mm}}$\\
			\relax\raggedright #1}}}

\newcommand{\bA}{\mathbf{A}}
\newcommand{\R}{\mathbb{R}}

\begin{document}

\title[Isoperimetric inequalities for magnetic Steklov problem]{Isoperimetric inequalities for the lowest magnetic Steklov eigenvalue}

\author[A.~Kachmar]{Ayman Kachmar}
\address{A. Kachmar\newline 
Department of Mathematics {\normalfont and} PDE Research Unit--Center for Advanced Mathematical Sciences\newline
American University of Beirut, P.O.Box 11-0236
Riad El-Solh, Beirut 1107 2020, Lebanon}

\email{ak292@aub.edu.lb}
\author[V.~Lotoreichik]{Vladimir Lotoreichik}
\address{V.~Lotoreichik\newline
	Department of Theoretical Physics\newline
	Nuclear Physics Institute, Czech Academy of Sciences, 
	25068, \v{R}e\v{z}, Czech Republic
}
\email{lotoreichik@ujf.cas.cz}
\subjclass{35P15, 49R05, 51M16}

\begin{abstract}
This paper studies the optimization of the lowest eigenvalue of the magnetic Steklov problem on planar domains. In the bounded domain setting and for magnetic fields of moderate strengths, we prove that among all simply-connected smooth domains of given area, the disk maximises the lowest magnetic Steklov eigenvalue. For exterior domains, we establish a similar isoperimetric inequality for magnetic fields of moderate strength under fixed perimeter constraint and additional geometric
and symmetry
assumptions. The proofs rely on the method of torsion-type trial functions in the bounded domain case and on the method of trial functions dependent only on the distance to the boundary in the exterior domain case.
\end{abstract}

\keywords{Isoperimetric inequality, torsion function, outer parallel curves, Dirichlet-to-Neumann operator, spectral shape optimization, magnetic Steklov problem}

\maketitle

\section{Introduction}

The Steklov eigenvalue problem, where the spectral parameter appears in the boundary condition, has been extensively studied in spectral geometry (see the review paper~\cite{GP17} and also~\cite[Chapter 7]{LMP}). Recently, its magnetic counterpart where the Laplacian is replaced by the magnetic Laplacian has attracted considerable attention (see \cite{CGH, EO, HN-e, HN, HKN}).

For bounded domains, optimization of spectral quantities under geometric constraints is a classical topic. The celebrated Faber--Krahn inequality states that among domains of fixed volume, the first Dirichlet Laplacian eigenvalue is minimized by the ball. Analogous isoperimetric inequalities have been established for the
first non-trivial Steklov eigenvalue without magnetic field~\cite{B01, BFNT21,W54}. In the magnetic setting, however, the interplay between geometry and the magnetic potential introduces new challenges. For the lowest eigenvalue of the magnetic Dirichlet Laplacian with homogeneous magnetic field on a planar domain an analogue of the Faber--Krahn inequality is proved in~\cite{E96}.  Recent works such as \cite{CPS, CLPS} and \cite{KL22, KL24} have developed techniques for handling shape optimization in magnetic spectral problems with Neumann and also Robin boundary conditions, including the technique of torsion-type trial functions. 

Spectral isoperimetric inequalities in exterior domains have been first studied in~\cite{KL18, KL20,KrL24} and in~\cite{B25} for the Robin Laplacian. More recent contributions address the magnetic Neumann eigenvalue problem~\cite{KLS} and the (non-magnetic) Steklov problem~\cite{BGGLP25} on exterior domains. In the magnetic setting, the Steklov problem on exterior domains has been studied recently in \cite{HN-e,HKN}, primarily in the context of spectral asymptotics with respect to the magnetic field strength.

Unlike in the non-magnetic case, the presence of the magnetic field makes the lowest Steklov eigenvalue to be non-zero and dependent on the shape of the domain. Optimization of the lowest magnetic Steklov eigenvalue in the case of the Aharonov-Bohm magnetic field was addressed in~\cite{CPS}. In this paper, we establish  isoperimetric inequalities for the lowest magnetic Steklov eigenvalue with homogeneous magnetic field in two geometric settings, showing that under moderate magnetic fields, the disk emerges as the optimal shape for maximizing the lowest magnetic Steklov eigenvalue in both bounded and exterior settings.

Firstly, for planar bounded simply-connected smooth domains of given area, we prove that the disk maximizes the eigenvalue provided the magnetic field strength is below a certain threshold that guarantees radial ground states for the disk. This result is stated as Theorem~\ref{prop:iso-DN}. It takes the form of a quantitative isoperimetric inequality and, as a corollary, implies that among all domains of fixed perimeter, the disk also maximizes the eigenvalue. This result is proved using the technique of torsion-type trial functions, that is, trial functions constant on the level lines of the torsion function. We employ in the proof some of the ideas developed in~\cite{CLPS} and also in~\cite{KL24} for the optimization of the lowest eigenvalue of the magnetic Neumann Laplacian. 

Secondly, for exterior domains (the complement of a planar bounded simply connected smooth region), we establish that among domains with a given perimeter and satisfying certain symmetry and geometric conditions (such as central symmetry  and a connectedness condition on outer parallel curves), the exterior of the disk maximizes the lowest magnetic Steklov eigenvalue. This holds when the magnetic field strength is below another critical constant guaranteeing a radial and real-valued ground state for the exterior of the disk. The assumption on the outer parallel curves is automatically satisfied in the exterior of a bounded convex domain. This result is stated as Theorem~\ref{thm:iso-fixed-L}.  In the proof of this result we use trial functions, which depend only on the distance to the boundary. This technique was pioneered in~\cite{PW61} and was employed later in many papers; see \eg~\cite{AFK17, FK15} for the Robin Laplacian with negative boundary parameter and \cite{KL22} for the magnetic Robin Laplacian.

The paper is organized as follows. Section~\ref{sec:bounded} treats bounded domains: after introducing the magnetic Steklov eigenvalue and analyzing the case of the disk, we develop the method of torsion-type trial functions and prove the isoperimetric inequality for fixed area. Section~\ref{sec:exterior} is devoted to exterior domains, where we establish a corresponding inequality under fixed perimeter and symmetry assumptions. For completeness, we also include appendices on the existence of ground states and the variational characterization of the lowest magnetic Dirichlet-to-Neumann eigenvalue.

\section{Bounded domains}\label{sec:bounded}
The aim of this section is to treat optimization of the lowest magnetic Steklov eigenvalue on a bounded planar domain. A central tool is the method of torsion-type trial functions, which reduces the problem to an auxiliary one-dimensional spectral problem. The behavior of this auxiliary problem under perturbations of a geometric weight function is analyzed in detail, showing that the functional involved is monotonic with respect to the weight.
\subsection{The lowest magnetic Steklov eigenvalue}
In this subsection, we will introduce the lowest magnetic Steklov eigenvalue on a bounded planar domain through its variational characterisation.
Suppose that $\Omega\subset\dR^2$ is a bounded, simply-connected domain with $C^\infty$-smooth boundary $\p\Omg$.  We denote the area of $\Omega$ by $|\Omega|$, and the perimeter of $\partial\Omega$ by $|\partial\Omega|$. By $\nu$ we denote the outer unit normal vector to the boundary of the domain $\Omg$. We will use the notation $L^2(\Omg)$ and $L^2(\p\Omg)$ for Lebesgue spaces on $\Omg$ and on its boundary $\p\Omg$, respectively. The first-order $L^2$-based Sobolev space on $\Omg$ will be denoted by $H^1(\Omg)$. The closure of $C^\infty_0(\Omg)$ in $H^1(\Omg)$ will be abbreviated by $H^1_0(\Omg)$. Finally, we will occasionally use the notation $H^{1/2}(\p\Omg)$ for the fractional Sobolev of order $\frac12$ on $\p\Omg$ and the notation $H^{-1/2}(\p\Omg)$ for its dual space, respectively.

Consider a vector potential $\bA\colon\Omega\to\R^2$ such that
\[\curl\bA=1.\]
We will make a specific choice $\bA=\bA_\Omega$,  that will prove useful in our analysis.  Let $\psi$ be the unique solution of the system
\begin{equation}\label{eq:torsion} 
\begin{cases}
-\Delta \psi=1, \quad &\mbox{in }\Omega,\\
\psi=0,\quad &\mbox{on }\partial\Omega,
\end{cases}
\end{equation}
in the Sobolev space $H^1_0(\Omg)$
and we set
\[\bA=\bA_\Omega:=(\partial_2\psi,-\partial_1\psi). \]
For a given $b>0$,  we introduce the lowest magnetic Steklov eigenvalue,
\begin{equation}\label{eq:def-ev}
	\lambda(b,\Omega):=\inf_{\substack{u\in H^1(\Omega)\\ u|_{\partial\Omega}\not=0}}\frac{\|(-\ii\nabla-b\bA)u\|^2_{L^2(\Omega;\dC^2)}}{\|u|_{\p\Omg}\|_{L^2(\partial\Omega)}^2},
\end{equation}
where $u|_{\p\Omg}$ stands for the trace of $u$ on $\p\Omg$.
\begin{remark}
The value $\lm(b,\Omg)$ can viewed as the lowest eigenvalue of the \emph{magnetic Dirichlet-to-Neumann map} on $\Omg$. Let us provide the definition of this operator. Recall that by~\cite[Lemmata 2.2 and 2.4, Definition 2.3]{BR12},
the \emph{Poisson operator} $\sfS_b\colon H^{1/2}(\p\Omg)\arr H^1(\Omg)$ defined by the properties 
\[
	(\sfS_b \varphi)|_{\p\Omg} = 
	\varphi,\qquad (-\ii\nb -b{\bf A})^2\sfS_b \varphi = 0,\qquad \text{for all}\,\, 
	\varphi\in H^{1/2}(\p\Omg),
\]
is continuous and maps $H^{1/2}(\p\Omg)$ bijectively onto
the space of solutions 
\begin{equation}\label{eq:cN}
	\cN_b:= 
	\big\{u\in H^1(\Omg)\colon (-\ii\nb -b{\bf A})^2u = 0\big\}.
\end{equation}
According to~\cite[Examples 2.1 and 2.2]{EO}, the quadratic form
\begin{equation}\label{eq:frlb}
		\frl_b[\varphi] := \|(-\ii\nb - b{\bf A})\sfS_b 
	\varphi\|^2_{L^2(\Omg;\dC^2)},\qquad \dom\frl_b := H^{1/2}(\p\Omg),
\end{equation}
is closed, non-negative, and densely defined in $L^2(\p\Omg)$.
Thus, the form $\frl_b$ defines by the first representation theorem a unique self-adjoint operator in the Hilbert space $L^2(\p\Omg)$, which we denote by $\Lm_b$. The operator $\Lm_b$ is called the magnetic Dirichlet-to-Neumann map.
Performing the integration by parts, 
we conclude that this operator is characterised by
\[
\begin{aligned}
	\Lm_b 
	\varphi &= \p_\nu\sfS_b \varphi|_{\p\Omg},\\
	\dom\Lm_b &= \big\{\varphi\in H^{1/2}(\p\Omg)\colon \p_\nu\sfS_b 
	\varphi|_{\p\Omg} \in L^2(\p\Omg)\big\},
\end{aligned}	
\]	
where we implicitly used that by~\cite[Section 3]{BR20} for any $u\in \cN_b$ the Neumann trace $\p_\nu u|_{\p\Omg} = \nu\cdot\nb u|_{\p\Omg}
\in H^{-1/2}(\p\Omg)$ is well defined.
Thanks to compactness of the embedding of $H^{1/2}(\p\Omg)$ into $L^2(\p\Omg)$, the spectrum of $\Lm_b$ is purely discrete. By Proposition~\ref{prop:lowest}, the value $\lm(b,\Omg)$ is the lowest eigenvalue of the self-adjoint operator $\Lm_b$.
\end{remark}
Let $u_\star\in H^1(\Omg)$ with $u_\star|_{\p\Omg}\ne 0$ be a function at which the infimum in~\eqref{eq:def-ev} is attained. Then $u_\star|_{\p\Omg}$ is a ground state of the magnetic Dirichlet-to-Neumann map on $\Omg$. With a slight abuse of terminology we also call the function $u_\star$ the ground state.
One can check that this function satisfies the system
\begin{equation}\label{eq:DN-gs}
	\begin{cases}
		\big(-\ii\nabla - b{\bf A}\big)^2u_\star = 0,&\quad\text{in}~\Omg,\\
		\p_\nu u_\star = \lm(b,\Omg) u_\star, &\quad\text{on}\,\p\Omg,
	\end{cases}
\end{equation}
where the Neumann trace $\p_\nu u_\star$ is taken with respect to the outer unit normal to $\Omg$ and is well-defined as the distribution in $H^{-1/2}(\p\Omg)$
(see Appendix~\ref{app:gs}). 
\subsection{The case of the disk}
We will establish in this subsection that for moderate magnetic fields the ground state for the disk is a radial function. To this aim we will employ separation of variables.
Let $\cB$ be a disk centered at the origin.  Denoting by $R$ the radius of $\cB$,  we have $R=\sqrt{|\cB|/\pi}$.   
\begin{prop}\label{prop:radial}
Let $\cB\subset\dR^2$ be the disk of radius $R > 0$.
If $b|\cB| < \pi$, then the infimum in~\eqref{eq:def-ev}
for $\Omg =\cB$ is attained on the radial function
\[u_\star(x)=f_\star(|x|),\]
where $f_\star\in C^\infty([0,R])$ satisfies the system
\begin{equation}\label{eq:Steklov_System}
\begin{cases}
 -f''_\star(r)-\frac1rf'_\star(r)+\frac{b^2r^2}{4}f_\star(r)=0\qquad\mbox{ for }0<r<R,\\
f'_\star(R)=\lambda(b,\cB) f_\star(R).
\end{cases}
\end{equation}
In particular, it holds (for some $A\in\dC\sm\{0\}$)
\[
	f_\star(r) = AI_0\left(\frac{br^2}{4}
\right),\qquad\lm(b,\cB) = \frac{bR}{2}
	\frac{I_1\left(\frac{bR^2}{4}\right)}{I_0\left(\frac{bR^2}{4}\right)}.
\]
\end{prop}
\begin{proof}
Using the explicit representation of the torsion function $\psi(x) = \frac{R^2-|x|^2}{4}$ for the disk, we conclude that the vector potential of the uniform magnetic field
acquires the form ${\bf A}(x) = \frac{b}{2}(-x_2,x_1)^\top$.
	
Any function $u\in H^1(\cB)$ can be uniquely decomposed into the series
\begin{equation}\label{eq:series}
	u(r,\tt) = \sum_{n\in\dZ} f_n(r)e^{\ii n\tt},
\end{equation}
where the measurable functions $f_n\colon (0,R)\arr\dC$ for $n\in\dZ$ satisfy the conditions
\begin{myenum}
	\item [\rm (a)] $\sum_{n\in\dZ} \int_0^R|f_n(r)|^2r\dd r < \infty;
	$ \item  [\rm (b)]  $\sum_{n\in\dZ}\int_0^R\left(|f_n'(r)|^2+\frac{n^2}{r^2}|f_n(r)|^2\right)r\dd r < \infty$.
\end{myenum}
Moreover, for any measurable functions $f_n\colon(0,R)\arr\dC$, $n\in\dZ$, satisfying the above conditions (a) and (b), the series in~\eqref{eq:series} converges to a function in the Sobolev space $H^1(\cB)$.
Plugging the function $u$, represented as the series~\eqref{eq:series}, into the Rayleigh quotient in~\eqref{eq:def-ev}, we eventually get
\begin{equation}\label{eq:characterisation}
\begin{aligned}
	\lm(b,\cB) &= \inf_{\begin{smallmatrix}
			u\in H^1(\cB)\\ u|_{\p\cB}\ne 0\end{smallmatrix}}\frac{\|(-\ii\nabla - b{\bf A})u\|^2_{L^2(\cB;\dC^2)}}{\|u|_{\p\cB}\|^2_{L^2(\p\cB)}}\\
	&=
	\inf_{n\in\dZ}\inf_{\begin{smallmatrix}
			f\colon f,f',n r^{-1}f\in 
			L^2((0,R);r\dd r)\\
			 f(R)\ne 0
		\end{smallmatrix}}	
	\frac{\displaystyle\int_0^R\left(|f'(r)|^2 + \left(\frac{n}{r}-\frac{br}{2}\right)^2|f(r)|^2\right)r \dd r}{R|f(R)|^2},
\end{aligned}		
\end{equation}
where $L^2((0,R);r\dd r)$ stands for the weighted $L^2$-space on $(0,R)$ with linear weight.
Consider the scalar potential on the interval $(0,R)$,
\[
v_n(r)=\left(\frac{n}{r}-\frac{br}{2}\right)^2,\qquad n\in\dZ.
\]
Under the assumption $b|\cB| < \pi$, which is equivalent to $b R^2 < 1$, the inequality $v_0 < v_n$ holds pointwise for all $n\in\dZ\sm\{0\}$. Hence, we get that the outer infimum in the characterisation above is attained for $n = 0$. Let $f_\star \in H^1((0,R);r\dd r)$ be the corresponding minimizing function for the inner infimum for $n = 0$, which exists 
and satisfies the system~\eqref{eq:Steklov_System}; see Appendix~\ref{app:gs-1D}. Then, we get that the infimum in~\eqref{eq:def-ev} for $\Omg = \cB$ is attained
on the radial function $u_\star(x) = f_\star(|x|)$. 

The general solution of the differential equation in the system~\eqref{eq:Steklov_System}
has the form
\[
	f(r) = AI_0\left(\frac{br^2}{4}\right) + BK_0 \left(\frac{br^2}{4}\right),\qquad A,B\in\dC.
\] 
Taking into account that $f_\star\in H^1((0,R);r\dd r)$ we get that $B = 0$ in view of properties of the modified Bessel function $K_0$ (see~\cite[10.29.3 and 10.30.2]{NIST}). Hence, we conclude that $f_\star(r) = AI_0\left(\frac{br^2}{4}\right)$ and,  as a direct consequence, we also get that
\[
	\lm(b,\cB) = \frac{f_\star'(R)}{f_\star(R)}= \frac{bR}{2}
	\frac{I_1\left(\frac{bR^2}{4}\right)}{I_0\left(\frac{bR^2}{4}\right)},
\]
where we used~\cite[10.29.3]{NIST}.
\end{proof}
The condition $b|\cB| <\pi$ is only sufficient for the ground state of the disk to be radial.
\begin{dfn}\label{def-criti-b}
We define the constant
\[ b_\star =\sup\{b>0\colon \lambda(b',\cB_1)\mbox{ has a radial ground state for all } b'\in(0,b)\},\]
where $\cB_1\subset\dR^2$ stands for the disk of the unit radius. 
\end{dfn}
This definition makes sense in light of Proposition~\ref{prop:radial},  and we have $b_\star\geq 1$.
\begin{remark}\label{rem:scaling}	
Using the unitary scaling transform $\sfU\colon L^2(\cB_1)\arr L^2(\cB_R)$
defined as
$(\sfU u)(x):= \frac{1}{R} u(\frac{x}{R})$, we get that,
for a general disk $\cB =\cB_R\subset\dR^2$ of radius $R >0$, the lowest magnetic Steklov eigenvalue $\lm(b,\cB)$ has a radial ground state provided that $b|\cB| < \pi b_\star$. In our analysis, we do not exclude that for larger values of $b$ the ground state of $\lm(b,\cB)$ becomes radial again, although the numerics in~\cite[Fig. 1]{HN} suggests that it is not the case. 
\end{remark}
\subsection{Method of torsion-type trial functions}
In this subsection, we apply the method of torsion-type trial functions to the magnetic Steklov problem. The technique was proposed simultaneously and independently in~\cite{KL24} and~\cite{CLPS}. 
We follow closely the approach in~\cite{CLPS}, adapting it to the Steklov problem.  

We introduce an auxiliary one-dimensional Steklov-type spectral problem, whose lowest eigenvalue provides an upper bound on the lowest magnetic Steklov eigenvalue on a general domain. This bound is an essential ingredient in proving an isoperimetric inequality for the magnetic Steklov problem with moderate magnetic fields.

Recall that $\psi$ stands for the torsion function of the domain $\Omg$ introduced as the unique solution of the system~\eqref{eq:torsion}. For the sake of brevity, the dependence of $\psi$ on the domain $\Omg$ is not indicated. 
Since the boundary of the domain $\Omg$ is $C^\infty$-smooth, standard elliptic regularity yields that $\psi \in C^\infty(\ov\Omg)$. Moreover, by the strong maximum principle, we get that $\psi > 0$ in $\Omg$. It is also argued in~\cite{CLPS} that $\psi$ is real-analytic in $\Omg$ and therefore it follows from~\cite[Proposition 1]{M20} applied to $|\nb\psi|^2$ that the set of its critical points $\{x\in\Omg\colon\nb\psi(x) = 0\}$ has Lebesgue measure zero.

Let $t_\star :=\max_{\overline\Omega}\psi$ and for all $t\in[0,t_\star]$,  we introduce the level and super-level sets 
\[ \Gamma_t:=\{x\in\overline\Omega\colon \psi(x)=t\},\qquad \Omega_t:=\{x\in \Omega\colon \psi(x)>t\}.\]
By Sard's theorem~\cite[Theorem 6.10]{Lee} and the implicit function theorem~\cite[Theorem 3.3.1]{KP02}, the curve $\G_t$ is $C^\infty$-smooth for almost all $t\in [0,t_\star]$.
We introduce the distribution function of $\psi$ by
\[ \mu(t):=|\Omega_t|,\]
which is right-continuous and strictly decreasing.
By~\cite[Corollary 2.2]{CLPS}
the function
\[
	\gamma_\Omg\colon[0,t_\star]\arr[0,\infty],\qquad\gamma_\Omg(t) := \int_{\G_t}\frac{1}{|\nb\psi(x)|}\dd\cH^1(x)
\]
is well defined almost everywhere and belongs to the space $L^1(0,t_\star)$; here $\cH^1$ stands for the one-dimensional Hausdorff measure. 
We introduce the parameter
\[
	a^\star =\mu(0)=|\Omega| 
\]
and for a given $a\in[0,a^\star]$, we define the value 
	\[ G_\Omega(a)=\int_{\Gamma_t}\frac{1}{|\nabla\psi(x)|}\dd\cH^1(x)\quad\mbox{ where }a=\mu(t), ~ t\in[0,t_\star].\]
In particular, at the endpoint $a = a^\star$,
\[G_\Omega(a^\star)=\int_{\partial\Omega}\frac{1}{|\nabla\psi(x)|}\dd\cH^1(x).\]
\begin{lem}[{\cite[Proposition 2.7]{CLPS}}]
	For almost every $a\in [0,a^\star]$ the inequality $G_\Omg(a) \ge 4\pi$ holds. The equality $G_\Omg(a) = 4\pi$ holds for almost every $a$ if, and only if, $\Omg$ is a disk.
\end{lem}
In order to state the next proposition we introduce the function space
\begin{equation}\label{key}
	\cF_\Omega:=\big\{ f\in L^2(0,a^\star)\colon f'\in  L^2((0,a^\star);aG_\Omega(a)\dd a)   \big\}.
\end{equation}

We have the following proposition, which is a variation of \cite[Proposition 2.11]{CLPS} adapted to the magnetic Steklov problem.
\begin{prop}\label{prop:ub}
For any $b > 0$, we have
\[ \lambda(b,\Omega)\leq \frac1{|\partial\Omega|}\kappa_1(b,G_\Omega),\]
where
\begin{equation}\label{eq:kappa}
	\kappa_1(b,G_\Omega)=\inf_{\begin{smallmatrix} f\in\cF_\Omg\\ f(a^\star)\not=0\end{smallmatrix}}\frac{\displaystyle\int_0^{a^\star}\Bigl(aG_\Omega(a)|f'(a)|^2+\frac{b^2a}{G_\Omega(a)}|f(a)|^2 \Bigr)\dd a}{|f(a^\star)|^2}.
\end{equation}
\end{prop}
\begin{proof}
For any $f\in\cF_\Omg$ with $f(a^\star) \ne 0$ we define the function
\[u(x)=f\bigl(\mu(\psi(x)) \bigr), \]
where as before $\psi$ is the torsion function and $\mu$ is its distribution function. By~\cite[Corollary 2.10]{CLPS} we have $u\in H^1(\Omg)$ and moreover $u|_{\p\Omg} \ne 0$.
Thus, we can use the function $u$ as the trial function in~\eqref{eq:def-ev}. By \cite[Eq. (27)]{CLPS}
we get
\[ \|(-\ii\nabla-b\bA)u\|^2_{L^2(\Omega;\dC^2)}=\int_0^{a^\star}\Bigl(aG_\Omega(a)|f'(a)|^2+\frac{b^2a}{G_\Omega(a)}|f(a)|^2\Bigr)\dd a.
\]
Moreover, by a direct computation we find
\[\|u|_{\p\Omg}\|_{L^2(\partial\Omega)}^2=\int_{\{\psi(x)=0\}}|u(x)|^2\dd \cH^1(x)=\int_{\partial\Omega}|f(a^\star)|^2\dd \cH^1(x)=|f(a^\star)|^2|\partial\Omega|,
\]
from which the desired bound follows.
\end{proof}
\begin{remark}\label{rem:disk}In the case of the disk,  Propositions~\ref{prop:radial},~\ref{prop:ub}, Definition~\ref{def-criti-b}
	and Remark~\ref{rem:scaling} yield
\[ \kappa_1(b,4\pi)=|\partial\cB|\lambda(b,\cB),\qquad\text{for}\,\, 0<b<\frac{\pi b_\star}{|\cB|}.\]
\end{remark}
\subsection{Auxiliary problem}
Let us consider two measurable bounded real-valued functions  $G_0$ and $G_1$ on $[0,a^\star]$ that satisfy  pointwise
\[ 4\pi \leq G_0(a)<G_1(a)\quad\mbox{a.e.}\]
Put
\[\delta G=G_1-G_0,\quad G_z=(1-z)G_0+zG_1,\quad \kappa(z)=\kappa_1(b,G_z)\quad(0\leq z\leq 1), \]
here $\kp_1(b,G_z)$ is defined as in~\eqref{eq:kappa} with $G_\Omg$ replaced by $G_z$.
Let us introduce the space\footnote{Note that $4\pi\leq G_0\leq G_z\leq G_1$ a.e. and by boundedness of $G_z$, the space
$\mathcal F$ coincides with the space in \eqref{key} with $G_\Omg$ replaced by $G_z$.} 
\begin{equation}\label{eq:def-space-F}
\cF := \big\{f\in L^2(0,a^\star)\colon f'\in L^2((0,a^\star);a\dd a)\big\}.
\end{equation}
Our aim in this subsection is to study the quantity
\begin{equation}\label{eq:kappa*}
\kappa(z) := \inf_{\begin{smallmatrix} f\in\cF\\ f(a^\star) \ne 0\end{smallmatrix}}\frac{\displaystyle \int_0^{a^\star}\left( a G_z(a)|f'(a)|^2 + \frac{b^2a}{G_z(a)}|f(a)|^2\right)\dd a}{|f(a^\star)|^2}
\end{equation}
and its dependence on the parameter $z$.
\begin{prop}\label{prop:der-kappa}
	For every $z\in[0,1]$,  
	we have
	\[
	\kappa'(z)=
	\int_0^{a^\star}
	\Bigl(|Y_z(a)|^2-b^2a^2|X_z(a)|^2 \Bigr)\frac{\delta G(a)}{a|G_z(a)|^2}\dd a,\]
	where
	\[ 
	Y_z(a)=aG_z(a) f'_z(a),\qquad X_z(a)=f_z(a),\]
	and $f_z\in\cF$ with $(aG_z f_z')'\in L^2(0,a^\star)$ is the unique solution of the system
	\begin{equation}\label{eq:fzsystem}
	\begin{cases}
		-\bigl(aG_zf'_z\bigr)'+\frac{b^2a}{G_z}f_z=0,\qquad\mbox{ on }[0,a^\star],\\
		\lim_{a\arr0}aG_z(a)f'_z(a)=0,\\
		a^\star G_z(a^\star) f'_z(a^\star)=\kappa(z) f_z(a^\star) ,\\
		f_z(a^\star)=1.
	\end{cases}
	\end{equation}
   Moreover, $f_z$ is locally absolutely continuous on $(0,a^\star]$ and $f_z(a)\not=0$ for all $a\in(0,a^\star)$.
\end{prop}
\begin{proof}
	Consider the family of closed, densely defined, symmetric and semibounded quadratic forms
	\[
	\frh_{z,\kp}[f] := \int_0^{a^\star}\left(
	a G_z(a)|f'(a)|^2 + \frac{b^2a}{G_z(a)}|f(a)|^2\right)\dd a-\kp|f(a^\star)|^2,\qquad \dom\frh_{z,\kp} := \cF,
	\]
	in the Hilbert space $L^2(0,a^\star)$;
	here $\kp\in\dR$ is the Robin parameter. The closedness of the form $\frh_{z,\kp}$ follows from the fact that the norm associated to the from $\frh_{z,\kp}$ can be easily shown to be equivalent to the natural norm of the space $\cF$ specified in~\cite[Lemma C.1]{CLPS}.
	A unique self-adjoint operator $\sfH_{z,\kp}$ in the Hilbert space $L^2(0,a^\star)$ is associated to the form $\frh_{z,\kp}$ by the first representation theorem~\cite[\S VI.2, Theorem 2.1]{Kato}.
	This operator can be viewed as
	the self-adjoint Sturm-Liouville operator
	associated with the differential expression
	\[-\frac{\dd}{\dd a}\left( aG_z 
	\frac{\dd}{\dd a}\right) + \frac{b^2a}{G_z}
	\]
	acting in the Hilbert space $L^2(0,a^\star)$	
	subject to the Robin boundary condition at the endpoint $a = a_\star$
	and the Neumann-type boundary condition at the endpoint $a = 0$.
	More precisely, adapting the analysis in~\cite[Appendix C]{CLPS} to our case, we infer that the action and the domain of $\sfH_{z,\kp}$
	can be characterised as
	\begin{equation}\label{key*}
	\begin{aligned}
		\sfH_{z,\kp}f &:= -(a G_zf')'+\frac{b^2a}{G_z}f,\\ 
        \dom\sfH_{z,\kp} &:= \Big\{f\in \cF\colon
		f, aG_zf' \in {\rm AC}(0,a^\star),
		(aG_z f')'\in L^2(0,a^\star), \\ &\qquad \qquad\qquad a^\star G_z(a^\star)f'(a^\star) = \kp f(a^\star),
		\lim_{a\arr0^+} a G_z(a) f'(a) = 0\Big\}.
	\end{aligned}	
	\end{equation}
	Since, by~\cite[Lemma C.1]{CLPS}, the space $\cF$ is compactly embedded into $L^2(0,a^\star)$, the operator $\sfH_{z,\kp}$ has purely discrete spectrum
	and its lowest eigenvalue admits the following variational characterisation		
	\begin{equation}\label{eq:Robin_var_char}
		\mu(z,\kp) := \inf_{f\in\cF\sm\{0\}}
		\frac{\displaystyle \int_0^{a^\star}\left(
			a G_z(a)|f'(a)|^2 + \frac{b^2a}{G_z(a)}|f(a)|^2\right)\dd a-\kp|f(a^\star)|^2}{\displaystyle \int_0^{a^\star}|f(a)|^2\dd a}.
	\end{equation}
 	It is straightforward to see that the lowest eigenvalue of $\sfH_{z,\kp}$ is simple and let us also denote by $f_{z,\kp}\in\cF$ the $L^2$-normalized non-negative eigenfunction of $\sfH_{z,\kp}$ corresponding to its lowest eigenvalue.
	Clearly, the function $f_{z,\kp}$ satisfies
	$f_{z,\kp}, (aG_zf_{z,\kp}')'\in L^2(0,a^\star)$ and is the solution of the system
	\begin{equation}\label{eq:fzkpsystem}
		\begin{cases}
			-(a G_z f_{z,\kp}')' + \frac{b^2a}{G_z} f_{z,\kp} = \mu(z,\kp)f_{z,\kp},\qquad\text{in}~~~(0,a^\star),\\[0.3ex]
			\lim_{a\arr 0} a G_z(a)f_{z,\kp}'(a) = 0,\\[0.3ex]
			a^\star G_z(a^\star) f_{z,\kp}'(a^\star) = \kp f_{z,\kp}(a^\star).
		\end{cases}	
	\end{equation}
	In particular, we infer that $f_{z,\kp}(a^\star)\ne 0$ as 
	otherwise it would hold $f_{z,\kp}\equiv 0$ by the unique continuation for the ordinary differential equation satisfied by $f_{z,\kp}$. 
	Since the family of the quadratic forms $\frh_{z,\kp}$
	in the Hilbert space $L^2(0,a^\star)$ forms a holomorphic family of the type~(a) in the sense of Kato~\cite[\S VII.4.1]{Kato} with respect to both parameters $z$ and $\kp$, we infer by~\cite[\S VII.4.6]{Kato} that the eigenvalue $\mu(z,\kp)$ is analytic in both arguments $z$ and $\kp$.
	Using the formula~\cite[Eq.~(4.56)]{Kato} we find that
	\begin{equation}\label{eq:muderivatives}
	\begin{aligned}
		\frac{\p\mu}{\p z} (z,\kp) &= \int_0^{a^\star}\left(a\dl G(a)|f'_{z,\kp}(a)|^2 - \frac{b^2a\dl G(a)}{G_z(a)^2}|f_{z,\kp}(a)|^2\right)\dd a,\\
		\frac{\p\mu}{\p \kp} (z,\kp) &= -|f_{z,\kp}(a^\star)|^2.
	\end{aligned}	 
	\end{equation}
	For a fixed, $z\in[0,1]$ the continuous function $\dR\ni\kp\mapsto\mu(z,\kp)$ is (strictly) decreasing in $\kp$.
	Note also that $\mu(z,0) > 0$ and $\lim_{\kp\arr+\infty}\mu(z,\kp) = -\infty$. Hence, we conclude that there exists a unique value $\hat\kp(z) > 0$ such
	that $\mu(z,\hat\kp(z)) = 0$. In a natural way $\kp(z)$ defined in~\eqref{eq:kappa*} is equal to $\hat\kp(z)$ for all $z\in[0,1]$. Indeed, using $f_{z,\hat\kp(z)}$ as the trial function in~\eqref{eq:kappa*} we get that $\kp(z)\le \hat\kp(z)$. If the strict inequality $\kp(z) < \hat\kp(z)$ would hold then we would be able to find for any $\eps > 0$ a function $g_{\eps}\in\cF$ such that the Rayleigh quotient in~\eqref{eq:kappa*} evaluated on this function $g_{\eps}$ is smaller that $\kp(z) + \eps$. Using this function $g_{\eps}$ with sufficiently small $\eps > 0$ as the trial function in the variational characterisation of $\mu(z,\hat\kp(z))$ we would get that $\mu(z,\hat\kp(z)) < 0$ leading to a contradiction.

	Applying the implicit function theorem to the equation $\mu(z,\kp(z)) = 0$ and using the expressions~\eqref{eq:muderivatives} we conclude that $\kp(z)$ is differentiable and its derivative is given by
 	\[
 		\kp'(z) = 
 		\frac{\displaystyle\int_0^{a^\star}\left(a\dl G(a)|f'_{z,\kp(z)}(a)|^2 - \frac{b^2a\dl G(a)}{G_z(a)^2}|f_{z,\kp(z)}(a)|^2\right)\dd a}{|f_{z,\kp(z)}(a^\star)|^2}.
 	\]
 	Setting $f_z := \frac{f_{z,\kp(z)}}{f_{z,\kp(z)}(a^\star)}$,
 	which in view of~\eqref{eq:fzkpsystem} satisfies~\eqref{eq:fzsystem},
 	we get the expression for $\kp'(z)$ as in the formulation of the proposition. Finally, $f_{z}$ never vanishes on $(0,a^\star)$ since if $f_{z}(a_0)=0$ with  $a_0\in(0,a^\star)$,  then taking the inner product 
    of $-\bigl(aG_zf'_z\bigr)'+\frac{b^2a}{G_z}f_z$ and $f_z$ in $L^2(0,a_0)$, we get after integrating by parts
    \[\int_0^{a_0}\Bigl(aG_z(a)|f'_z(a)|^2+\frac{b^2a}{G_z(a)}|f_z(a)|^2\Bigr)\dd a=0,\]
    which yields that $f_z\equiv0$ on $(0,a_0)$, and by the unique continuation for the ordinary differential equation satisfied by $f_z$, we get $f_z\equiv0$ on $(0,a^\star)$. 
    More precisely, arguing as in \cite[Appendix C]{CLPS}, $X_z(a)=f_z(a)$ and $Y_z(a)=aG_z(a)f_z'(a)$ are absolutely continuous on $(a_0/2,a^\star]$, and $\Psi(a):=(X_z(a),Y_z(a))^\top$ satisfies 
    \[\begin{cases}
        \Psi'=\mathbf F(a,\Psi)&\mbox{on }(0,a^\star)\\
        \Psi(a_0)=0,
    \end{cases}\]
    where 
    \[\mathbf F(a,\Psi)=\begin{pmatrix}0&\frac{1}{aG_z(a)}\\
    \frac{b^2a}{G_z(a)}&0\end{pmatrix}\Psi\]
    is Lipschitz in $\Psi$ uniformly on $(a_0/2,a^\star]$. By Carath\'eodory's existence and uniqueness theorem (see \cite[Theorem~5.3, p. 30]{Ha}), we get that $\Psi\equiv 0$ on $[a_0,a^\star]$. 
\end{proof}
In the next proposition we show that
 the integrand in the formula 
for $\kp'(z)$ in Proposition~\ref{prop:der-kappa} is pointwise negative. In the proof of this proposition we use the method inspired by the proof of~\cite[Proposition 3.8]{CLPS}.
\begin{prop}
	\label{prop:sign-der}
	Let $\kp(z)$ be defined as in~\eqref{eq:kappa*} and $Y_z,X_z$ be defined as in Proposition~\ref{prop:der-kappa}.
	Then it holds that
	\[|Y_z(a)|^2-b^2a^2|X_z(a)|^2<0\qquad\mbox{ for all } a\in(0,a^\star).\]
\end{prop}
\begin{proof} 
	In the proof, we will frequently omit indicating dependence on $z$ to simplify the notation.
	Since, by the previous proposition, $f_z$ is a continuous non-vanishing function on $(0,a^\star]$
	satisfying $f_z(a^\star) = 1$, we get that $X_z > 0$ on $(0,a^\star)$.
	We introduce the function $R(a)=Y_z(a)/X_z(a)$.
	We will demonstrate that this function is positive on $(0,a^\star)$
	and in order to get the claim of the proposition it would suffice to show that $R(a)<ba$ for all $a\in(0,a^\star)$.
	
	\smallskip
	
	\noindent {\it Step 1.} We introduce the notation $P(a):=aG_z(a)$ and $Q(a):=a/G_z(a)$. 
	Like \cite[Proposition~3.11]{CLPS}, we deduce from~\eqref{eq:fzsystem} that the 
	functions $X_z(a)$ and $Y_z(a)$ satisfy the system
	\[
	\begin{cases}
		X' = \frac{1}{P} Y,\\
		Y' = b^2Q X,\\
		Y(a^\star) = \kp X(a^\star),\\
		\lim_{a\arr0^+}Y(a)=0,
	\end{cases}	
	\]
	where the dependence on $z$ was omitted. 
	We deduce from the equation $Y' = b^2QX$ combined with positivity of $X$ on
	the whole interval $(0,a^\star)$ that $Y$ is (strictly) increasing on $(0,a^\star)$. This monotonicity property and the boundary condition $\lim_{a\arr0^+}Y(a) = 0$ yield that $Y > 0$ on $(0,a^\star)$. Hence, from the equation $X' = \frac{1}{P}Y$ we get that $X$ is (strictly) increasing on $(0,a^\star)$ and thus it extends by continuity to the endpoints of this interval. In particular, we have $X(0)\ge 0$ and $X(a^\star) > 0$.

	The function $R$ is thus positive and continuous on $(0,a^\star]$.
	Moreover, we can estimate this function from above as	
	\begin{equation}\label{eq:Rbnd}
	\begin{aligned}
		R(a) & = \frac{b^2\displaystyle\int_0^a Q(t)X(t)\dd t}{X(a)} \le 
		b^2\int_0^a \frac{t}{4\pi}\dd t = \frac{b^2a^2}{8\pi},
	\end{aligned}	
	\end{equation}
	where we used the identity $Y' = b^2QX$ and the boundary condition $\lim_{a\arr0^+}Y(a) = 0$ in the first step
	and monotonicity of $X$ and the inequality $G_z \ge 4\pi$ in the second step. Thus, we conclude from this bound that $R$ extends by continuity to the point $a = 0$ and it holds that $R(0) = 0$.
	In particular, we infer from~\eqref{eq:Rbnd} that there exists $a_0\in (0,a^\star)$ such that $R(a) < ba$ for all $a\in (0,a_0)$.

	By a straightforward computation, $R$ solves
	the system
	\begin{equation}\label{eq:R}
	\begin{cases}
		R' = F(a,R),\\
		R(a^\star)=\kp > 0,\\
		 R(0)=0,
	\end{cases}
	\end{equation}
	where 
	\[
	F(a,Z):=b^2Q(a)-\frac1{P(a)}Z^2.
	\]
	
	\smallskip 
	
	\noindent {\it Step 2.}
	We will prove that $R(a)\leq b a$ for all $a\in(0,a^\star)$. If this were not the case, then there would exist $a\in(0,a^\star)$ such that $R(a)>b a$. Let $a_1=\inf\{ a\in (0,a^\star)\colon R(a)>b a\}$. Since
		we have already checked that $R(a) < ba$ for all sufficiently small $a >0$, we conclude $0<a_1<a^\star$, and by continuity, $R(a_1)=b a_1$ and $R(a)<ba$ for all $a\in (0,a_1)$.
	We will prove that there exists $a\in(0,a_1)$ such that $R(a)>ba$, thereby obtaining a contradiction.
	
	We write
	\[R(a)-b a=\int_a^{a_1}\bigl(b-F(\alpha,R(\alpha))\bigr)\dd\alpha. \]
	Note that we have $F(a_1,R(a_1)) =0$.
   We choose $\varepsilon>0$ sufficiently small so that for all $\alpha\in[a_1-\varepsilon,a_1)$,
	\[b-F(\alpha,R(\alpha))>0.\]
	Consequently, for $a\in[a_1-\varepsilon,a_1]$, we have
	\[R(a)-ba>0,\]
	which violates the definition of $a_1$.
	\end{proof}

\subsection{Isoperimetric inequality}

Now we are in position to formulate and prove an isoperimetric inequality for the magnetic Steklov eigenvalue on a bounded domain.

\begin{thm}
	\label{prop:iso-DN}
	Let $\Omg\subset\dR^2$ be a bounded simply-connected $C^\infty$-smooth domain. Let the constant $b_\star > 0$ be as in Definition~\ref{def-criti-b}. Suppose that $\Omega$ is not a disk and that $0<b|\Omg|<\pi b_\star$.
	Let $\cB$ be the disk of the same area as $\Omg$.
	Let $\lm(b,\Omg)$ and $\lm(b,\cB)$ be the lowest magnetic Steklov eigenvalues defined as in~\eqref{eq:def-ev}
	and associated with $\Omg$ and $\cB$, respectively.
	Then the following isoperimetric inequality holds
	\[|\p\Omg|\lambda(b,\Omega)<|\p\cB|\lambda(b,\cB).\]
	In particular, by the geometric isoperimetric inequality, $\lm(b,\Omg) < \lm(b,\cB)$ holds.
\end{thm}
As a corollary to Theorem~\ref{prop:iso-DN}, the disk is also a maximizer under fixed perimeter.
\begin{cor}\label{corol:iso-DN}
Under the assumptions of Theorem~\ref{prop:iso-DN}, it holds
\[\lambda(b,\Omega)<\lambda(b,\cB'),\]
where $\cB'\subset\dR^2$ is the disk with the same perimeter as $\Omega$.
\end{cor}
\begin{proof} Let $t=|\partial\Omega|/\sqrt{4\pi|\Omega|}=|\partial\cB'|/|\partial\cB|$.  By the geometric isoperimetric inequality, $t>1$, and by scaling, we have $\lambda(b,\cB') =t^{-1}\lambda(bt^2,\cB)$. Thanks to \cite[Theorem 1.3]{HN}, we know that the function $\dR_+\ni\beta\mapsto\lambda(\beta,\cB)$ is increasing. Thus, 
as a consequence of Theorem~\ref{prop:iso-DN}, it holds
\[\lambda(b,\Omega)<\frac{|\p\cB|}{|\p\Omg|}\lm(b,\cB)
\le
t^{-1}\lm(bt^2,\cB) = \lm(b,\cB').\qedhere\]
\end{proof}
\begin{proof}[Proof of Theorem~\ref{prop:iso-DN}]
	We will prove that $\kappa_1(b,G_\Omega) < \kp_1(b,4\pi)$. Once this is proved, we can deduce from
	Proposition~\ref{prop:ub} combined with the assumption on $b$ and Remark~\ref{rem:disk}
	that the inequality $|\p\Omg|\lambda(b,\Omega)<|\p\cB|\lambda(b,\cB)$ holds.
	The inequality $\lambda(b,\Omega)<\lambda(b,\cB)$ follows as a consequence taking into account
	the standard geometric isoperimetric inequality $|\p\Omg| > |\p\cB|$. 
	
	\smallskip
	
	\noindent\emph{Step 1.}
	Since the function $G_\Omg$ is not necessarily bounded, we aim at reducing to bounded weight functions using an approximation procedure.
	Let us introduce the sequence of functions $G_n\colon(0,a^\star)\arr\dR_+$,  $G_n :=\min\{G_\Omega,4\pi n\}$, $n\in\N$. We will prove that
	\[\lim_{n\to\infty}\kappa_1(b,G_n)=\kappa_1(b,G_\Omega),\]
	where $\kp_1(b,G_n)$ is characterized  as in~\eqref{eq:kappa} with $G_\Omg$ replaced by $G_n$ and $\mathcal F_\Omg$ replaced by the space $\mathcal F$ defined in \eqref{eq:def-space-F}.
	We have $4\pi\leq G_1\leq G_2\leq \cdots$ and that $G_n\to G_\Omega$ pointwise.
	By Proposition~\ref{prop:der-kappa} combined with Proposition~\ref{prop:sign-der} the non-negative sequence $(\kp_1(b,G_n))_n$ is non-increasing. Thus, this sequence is convergent. 
	
	Consider $\eta>0$ and choose $f\in\cF_\Omg$ with $f(a^\star) = 1$ such that
	\[		\int_0^{a^\star}\Bigl(aG_\Omega(a)|f'(a)|^2+\frac{b^2a}{G_\Omega(a)}|f(a)|^2\Bigr)\dd a\leq \kp_1(b,G_\Omg) + \eta.
	\]
	We have
	\[\begin{gathered}
		\lim_{n\to\infty}\int_0^{a^\star}aG_n(a)|f'(a)|^2\dd a=
		\int_0^{a^\star}aG_\Omega(a)|f'(a)|^2\dd a,\\ \lim_{n\to\infty}\int_0^{a^\star}\frac{a}{G_n(a)}|f(a)|^2\dd a
		=\int_0^{a^\star}\frac{a}{G_\Omega(a)}|f(a)|^2\dd a,
	\end{gathered}\]
	by the monotone and dominated convergence, respectively. Consequently,
	\[\kappa_1(b,G_\Omega)\geq \lim_{n\to\infty}\kappa_1(b,G_n)-\eta,\]
	and by taking $\eta\to0^+$, we get
	\begin{equation}\label{eq:kappa_ineq1}
	\kappa_1(b,G_\Omega)\geq \lim_{n\to\infty}\kappa_1(b,G_n).
	\end{equation}
	For a given $n\geq 1$, let $g_n\in \cF = \{h\in L^2(0,a^\star)\colon h'\in L^2((0,a^\star);a\dd a)\}$ be a ground state 
	at which the infimum in the characterisation of
	$\kappa_1(b,G_n)$ is attained and normalized so that $g_n(a^\star)=1$. Using that the inequality $G_n\ge 4\pi$ holds pointwise for all $n\in\dN$, that $G_n$ is pointwise non-decreasing in $n$, and that $(\kp_1(b,G_n))_n$ is a non-increasing sequence, we conclude that $g_n'$ is a bounded sequence in $L^2((0,a^\star);a\dd a)$. Moreover, we find using absolute continuity of $g_n$ that
	\[
	\begin{aligned}
		\int_0^{a^\star} |g_n(a)|^2\dd a& = \int_0^{a^\star}\left| g_n(a^\star) -\int_a^{a^\star} g_n'(\hat{a})\dd \hat{a}\right|^2\dd a\\
		&\le
		2a^\star + 2\int_0^{a^\star}
		\left(\int_{a}^{a^\star}\frac{1}{\hat{a}}\dd\hat{a}\right)\left( \int_a^{a^\star}\hat{a}|g_n'(\hat{a})|^2\dd \hat{a}\right)\dd a\\
		&\le
		2a^\star + 2\left( \int_0^{a^\star}\hat{a}|g_n'(\hat{a})|^2\dd \hat{a}\right)\left(\int_0^{a^\star}
		\ln\left(\frac{a^\star}{a}\right)\dd a\right)\\
		&\le
		2a^\star\left(1+\int_0^{a^\star}\hat{a}|g_n'(\hat{a})|^2\dd \hat{a}\right).
	\end{aligned}	
	\]
	Thus, $(g_n)_n$ is a bounded sequence in $L^2(0,a^\star)$ and hence also a bounded sequence\footnote{$\mathcal F$ is endowed with the norm $\|h\|_{\mathcal F}=\|h\|_{L^2(0,a^\star)}+\|h'\|_{L^2((0,a^\star);a\dd a)}$.} in $\cF$.
	Passing to a subsequence and using compactness of the embedding of $\cF$ into $L^2(0,a^\star)$ (see~\cite[Lemma C.1]{CLPS}), we may suppose that there exists $g_\star\in\cF$ and a growing sequence $(n_k)_k$ such that
	\begin{equation}\label{eq:convergence}
	g_{n_k}\to g_\star\qquad\text{weakly in}~~ \cF~~\text{and strongly in}~~~L^2(0,a^\star),\qquad\text{as}~k\arr\infty.
	\end{equation} 
	Moreover, the equality $g_\star(a^\star) = 1$ holds since $H^1(a^\star/2,a^\star)$  is compactly embedded in $C([a^\star/2,a^\star])$ (see \cite[Theorem~8.8]{B}), and $(g_{n_k}|_{(a^\star/2,a^\star)})_k$ is bounded in $H^1(a^\star/2,a^\star)$.  
    
	Consequently, for $n_k \ge m$, we get
	\[
	\begin{aligned}	
		&\int_0^{a^\star} a G_m(a)|g_{n_k}'(a)|^2\dd a + \int_0^{a^\star}\frac{b^2a}{G_{n_k}(a)}|g_{n_k}(a)|^2\dd a\\
		&\qquad\le 
		\int_0^{a^\star} a G_{n_k}(a)|g_{n_k}'(a)|^2\dd a + \int_0^{a^\star}\frac{b^2a}{G_{n_k}(a)}|g_{n_k}(a)|^2\dd a =\kp_1(b,G_{n_k}).
	\end{aligned}
	\]
	We derive from~\eqref{eq:convergence} that
	$g_{n_k}'\arr g_\star'$ weakly in $L^2((0,a^\star);a\dd a)$.
	Hence, we infer that
	\[
		\int_0^{a^\star} a  g_{n_k}'(a)\ov{\varphi(a)}\dd a\arr \int_0^{a^\star} a g'_\star(a)\ov{\varphi(a)}\dd a,\qquad\text{as}\quad k\arr\infty,
	\]
	for any $\varphi\in L^2((0,a^\star);a\dd a)$. For $\varphi = G_mg_\star'$ we get
	\[
		\int_0^{a^\star} a G_m(a) g_{n_k}'(a)\ov{g_\star'(a)}\dd a\arr \int_0^{a^\star} a G_m(a) |g'_\star(a)|^2\dd a,\qquad\text{as}\quad k\arr\infty.
	\]
	Using  Cauchy-Schwarz inequality we derive from the above limit that
	\begin{equation}\label{eq:lim1}
	\liminf_{k\arr\infty} \int_0^{a^\star} a G_m(a)|g_{n_k}'(a)|^2\dd a \ge \int_0^{a^\star} a G_m(a)|g_\star'(a)|^2\dd a.
	\end{equation}
	Moreover, we find that
	\begin{equation}\label{eq:lim2}
	\begin{aligned}
		&\int_0^{a^\star} 
		\frac{b^2 a}{G_{n_k}(a)}|g_{n_k}(a)|^2\dd a\\
		&\qquad=
		\int_0^{a^\star} 
			\frac{b^2 a}{G_\Omg(a)}|g_\star(a)|^2\dd a+
		\int_0^{a^\star} 
				\left(\frac{b^2 a}{G_{n_k}(a)} - \frac{b^2 a}{G_\Omg(a)}\right)|g_\star(a)|^2\dd a\\
		&\qquad\qquad+
		\int_0^{a^\star} 
						\frac{b^2 a}{G_{n_k}(a)}\left(|g_{n_k}(a)|^2-|g_\star(a)|^2\right)\dd a	\arr
								\int_0^{a^\star} 
									\frac{b^2 a}{G_\Omg(a)}|g_\star(a)|^2\dd a,\quad k\arr\infty,
	\end{aligned}	
	\end{equation}
	where the second term on the right-hand side converges to zero by Lebesgue dominate convergence theorem and the third term tends to zero since $g_{n_k}$ converges to $g_\star$ in $L^2(0,a^\star)$.
	
	Using~\eqref{eq:lim1} and ~\eqref{eq:lim2} and passing to the limit $k\arr\infty$, we infer that
	\[
	\int_0^{a^\star} a G_m(a)|g_\star'(a)|^2\dd a + \int_0^{a^\star}\frac{b^2a}{G_\Omg(a)}|g_\star(a)|^2\dd a \le \lim_{n\arr\infty}\kp_1(b,G_{n}).
	\]
	Since the above inequality is true for any $m\in\dN$, we end up with $g_\star\in\cF_\Omg$ and moreover, $\kp_1(b,G_\Omg) \le \lim_{n\arr\infty}\kp_1(b,G_n)$.
	
	\noindent {\it Step 2.}
	First note that, by Step 1,  $\kp_1(b,G_\Omg)$ is the limit of a non-increasing sequence $(\kp_1(b,G_n))_n$.
	Second, a combination of Proposition~\ref{prop:der-kappa} and Proposition~\ref{prop:sign-der} yields that $\kp_1(b,G_n) < \kp_1(b,4\pi)$ for all sufficiently large $n\in\dN$, from which the desired inequality	$\kp_1(b,G_\Omg) < \kp_1(b,4\pi)$ follows.
\end{proof}

\section{Exterior domains}\label{sec:exterior}

In this section, we introduce the lowest magnetic Steklov eigenvalue on an exterior domain and prove an isoperimetric inequality for this eigenvalue under fixed perimeter constraint and additional geometric assumptions on the domain. Our analysis uses geometric properties of outer parallel curves and a comparison of moments, inspired by \cite{KL22} and classical isoperimetric arguments.

\subsection{The lowest magnetic Steklov eigenvalue}
Consider $\Omg^\ext=\mathbb R^2\setminus\overline{\Omega}$, the exterior of a bounded simply-connected $C^\infty$-smooth domain $\Omg\subset\dR^2$. We introduce the vector potential 
in the symmetric gauge
\begin{equation}\label{eq:Asymm}\bA_\circ(x)=\frac12(-x_2,x_1)^\top\end{equation}
and for every $b>0$, we introduce the space
\[X_b(\Omg^\ext)=\{u\in L^2(\Omg^\ext)\colon (-\ii\nabla-b\bA_\circ)u\in L^2(\Omg^\ext;\dC^2)\}.\]
The symmetric gauge is convenient for our subsequent analysis.
We introduce the lowest magnetic Steklov eigenvalue in $\Omega^\ext$ directly by its variational characterisation in full analogy with the case of bounded domains
\begin{equation}\label{eq:def-ev-ext}
\lambda(b,\Omg^\ext)=\inf_{\substack{u\in X_b(\Omg^\ext)\\u|_{\p\Omg}\not=0}}\frac{\|(-\ii\nabla-b\bA_\circ)u\|_{L^2(\Omg^\ext;\dC^2)}^2}{\|u|_{\p\Omg}\|_{L^2(\partial\Omg)}^2}.
\end{equation}
The quantity $\lambda(b,\Omg^\ext)$ is well defined, finite, and non-negative. As in the case of bounded domains, this eigenvalue can be identified with the lowest eigenvalue of the magnetic Dirichlet-to-Neumann map on the exterior domain $\Omg^\ext$; see~\cite[Section 2]{HN-e} and also~\cite{HKN} for details.   
We also remark that by~\cite[Proposition 2.2]{HN-e} the lowest magnetic Steklov eigenvalue $\lambda(b,\Omg^\ext)$ is not equal to zero for any $b > 0$. We remark that the exterior domain
is not simply-connected and the quantity $\lm(b,\Omg^\ext)$ depends, in general, on the choice of the gauge of the magnetic field\footnote{To fix the gauge, one would need to prescribe the  magnetic flux as in \cite[Eq. (1.4)]{HKN2}).}. Our further results are proved for the symmetric gauge~\eqref{eq:Asymm}. 

A function $u_\star\in X_b(\Omg^\ext)$ with $u_\star|_{\p\Omg} \ne 0$ at which the infimum in~\eqref{eq:def-ev-ext} is attained will be called the ground state of the magnetic Steklov problem. It is not necessary for our analysis to show that the ground
state always exists, although we expect it to be the case. With a slight abuse of terminology, we will say that $u_\star$ is a ground state for $\lm(b,\Omg^\ext)$.
The following result follows from \cite[Proposition~1.2]{HN-e}.
\begin{prop}\label{prop:HN}
    There exists $b_1>0$ such that, for $0<b<b_1$, the lowest magnetic Steklov eigenvalue in the exterior of the unit disk, $\lambda(b,\cB_1^\ext)$, has a real-valued radially symmetric ground state. 
\end{prop}
Thanks to the above proposition, we may introduce the positive constant
\begin{equation}\label{eq:def-b-circ}
b_\circ\!:=\!\sup\{b'>0\colon\! \lambda(b,\cB_1^\ext)\mbox{ has a real-valued and radial ground state for all }b\!\in\!(0,b')\}.
\end{equation}

\subsection{Isoperimetric inequality}
In the case of the exterior of a bounded simply-connected domain with fixed perimeter satisfying some additional geometric assumptions, we prove that the lowest magnetic Steklov eigenvalue is maximized by the exterior of the disk
provided that the intensity of the magnetic field does not exceed a critical constant, which guarantees that the corresponding magnetic Steklov problem on the exterior of the disk has a radial and real-valued ground state.
Before we state this result we need to recall several geometric definitions.
\begin{dfn}
	We say that the domain $\Omg$ is \emph{centrally symmetric} if there exists $x_0\in\dR^2$ such that $\Omg = \{x_0-x\in\dR^2 \colon x\in \Omg-x_0\}$. In this case, we say that $\Omg$ is centrally symmetric with respect to $x_0$.
\end{dfn}
Next, we introduce the distance function to $\Omg$
\[
	\rho_{\Omg}\colon\dR^2\arr\dR_+,\qquad \rho_{\Omg}(x) := \inf_{y\in\Omg}|x-y|.
\]
This function is identically equal to zero  for all $x\in\Omg$ and positive for all $x\in\Omg^\ext$.
\begin{dfn}
	For a bounded simply-connected $C^\infty$-smooth domain $\Omg\subset\dR^2$ we introduce the outer parallel curve
	on the distance $t> 0$ to the boundary by
	\[
		\Sg_t = \Sg_t(\Omg) := \{x\in\dR^2\colon \rho_{\Omg}(x) = t\}.
	\]
\end{dfn}
By~\cite[Theorem 4.4.1]{SST}, the outer parallel curve $\Sg_t$ is at most a finite union of piecewise $C^\infty$-smooth simple closed curves for almost all $t\in\dR$. The curve $\Sg_t$ may have, in general, several connected components. However, if $\Omg$
is a convex domain, then it is easy to see that $\{x\in\dR^2\colon \rho_\Omg(x) < t\}$ is also a convex domain and thus the outer parallel curve $\Sg_t$ is connected being its boundary. By~\cite[Proposition A.1]{S01}
it holds that 
\begin{equation}\label{eq:length_bnd}
	|\Sg_t| \le |\p\Omg| +2\pi t,\qquad \text{for all}\,\,t> 0,
\end{equation}
where the equality is obviously attained for all $t>0$ when $\Omg$ is a disk. 
\begin{remark}
One can show that the equality in~\eqref{eq:length_bnd} is attained for all $t>0$  when $\Omg$ is convex. This stronger property is not needed for our purposes.
\end{remark}
\begin{remark}
	In the above discussion, we used~\cite[Theorem 4.4.1]{SST}
	and~\cite[Proposition A.1]{S01}, which are only proved for bounded domains and interior parallel curves. However, the desired properties of outer parallel for exterior domains can be still deduced from these results by applying them to the intersection of $\Omg^\ext$ with a disk of a sufficiently large radius.
\end{remark}

Now we are in position to formulate and prove our main result on optimization of the lowest magnetic Steklov eigenvalue for exterior domains. 
\begin{thm}\label{thm:iso-fixed-L}
    Let $\Omg\subset\dR^2$ be a bounded simply-connected $C^\infty$-smooth domain such that the following holds:
    \begin{itemize}
     \item [\rm (a)] $\Omg$
    is either centrally symmetric or has two (not necessarily orthogonal) axes of symmetry.
    \item [\rm (b)] The outer parallel curve $\Sg_t = \Sg_t(\Omg)$ is a simple closed curve for all $t > 0$ (this condition is automatically satisfied when $\Omg$ is convex).
    \end{itemize}
 	Let the constant $b_\circ > 0$ be as in \eqref{eq:def-b-circ}. Suppose that $\Omega$ is not a disk and that $0<|\partial\Omg|^2b<4\pi^2 b_\circ$.
    Let $\cB'$ be the disk of the same perimeter as $\Omg$.
	Let $\lm(b,\Omg^\ext)$ and $\lm(b,(\cB')^\ext)$ be the lowest magnetic Steklov eigenvalues defined as in~\eqref{eq:def-ev-ext}
	and associated with $\Omg^\ext$ and $(\cB')^\ext$, respectively.
	Then the following isoperimetric inequality holds
	\[\lambda(b,\Omega^\ext)<\lambda(b,(\cB')^\ext).\]
\end{thm}

\begin{proof}
We split the proof into two steps.

\medskip

\noindent {\it Step 1. Moments of outer parallel curves.}
Without loss of generality we can assume that
the centre of symmetry of $\Omg$ is at the origin
if $\Omg$ is centrally symmetric (respectively, the intersection of the axes of symmetry is at the origin if $\Omg$ has two axes of symmetry). We also assume that the disk $\cB'$ is centred at the origin.
 
Let  $L=|\partial\Omg|$ be the perimeter of $\Omg$. 
Analogously to~\cite[Lemma 2.5]{DKL}, it follows that the centroid $c(t) = \frac{1}{|\Sg_t|}\int_{\Sg_t} x \dd\cH^1(x)$ of $\Sg_t$
is at the origin for all $t > 0$.

Let $\mathcal C_t=\Sigma_t(\cB')$ be the circle centered at the origin and of radius $(2\pi)^{-1}L+t$. Thanks to \eqref{eq:length_bnd} and a classical result by Hurwitz \cite{Hu} (see also~\cite[Proof of Theorem 1.1]{DKL}) we have
\begin{equation}\label{eq:moment}
\int_{\Sigma_t}|x|^2\dd\cH^1(x)\leq \int_{\mathcal C_t}|x|^2\dd\cH^1(x),
\end{equation}
with strict inequality when $\Sigma_t$ is not a circle (if $\Omega$ is not a disk, then $\Sigma_t$ is not a circle).

\medskip

\noindent {\it Step 2. Construction of a trial state.}

The definition of our trial state is inspired by \cite[Proof of Theorem 4.8]{KL22}. Let ${u_\circ\in X_b((\cB')^\ext)}$ be a real-valued and radial ground state of $\lambda(b,(\cB')^{\ext})$, whose trace on $\p\cB'$ is normalized to one in $L^2(\partial\cB')$. In the exterior of $\cB'$, the distance function of $\cB'$ is $r-R'$, where $R'=L/2\pi$ is the radius of $\cB'$ and $r=|x|$ is the radial variable. Thus, we can express $u_\circ$ as
\[u_\circ(x)=\psi_\circ\bigl(\rho_{\cB'}(x)\bigr),\qquad x\in (\cB')^\ext,\]
with some $\psi_\circ\colon \dR_+\arr\dR_+$. The representation of $\psi_\circ$  given in~\cite[Section 4]{HN-e} in terms of special functions yields, in particular, that $\psi_\circ\in C^\infty([0,\infty))$.
Now we define the trial state on $\Omg^\ext$ by
\[u_\star(x)=\psi_\circ\bigl(\rho_{\Omg}(x)\bigr),\qquad x\in\Omg^\ext.\]
Using that $u_\circ$ is real-valued and the co-area formula (see \cite[Eqs. (4.4) and (4.5)]{KL22}), we have
\[
\begin{aligned}
\|(-\ii\nabla-b\bA_\circ)u_\star\|^2_{L^2(\Omg^\ext;\dC^2)}&=\|\nabla u_\star\|^2_{L^2(\Omg^\ext;\dC^2)}+\frac{b^2}{4}\int_{\Omega^\ext}|x|^2|u_\star|^2\dd x\\
&=\int_0^{\infty}|\psi_\circ'(t)|^2|\Sg_t|\dd t+\frac{b^2}{4}\int_0^{\infty}|\psi_\circ(t)|^2\int_{\Sigma_t}|x|^2\dd\cH^1(x)\dd t.
\end{aligned}\]
Using~\eqref{eq:length_bnd} and \eqref{eq:moment}, we deduce 
\begin{equation}\label{eq:gradu}\begin{aligned}
    \|(-\ii\nabla-b\bA_\circ)u_\star\|^2_{L^2(\Omg^\ext;\dC^2)}&<\int_0^{\infty}|\psi_\circ'(t)|^2|\cC_t|\dd t+\frac{b^2}{4}\int_0^{\infty}|\psi_\circ(t)|^2\int_{\mathcal C_t}|x|^2\dd\cH^1(x)\dd t\\
    &=\|(-\ii\nabla-b\bA_\circ)u_\circ\|_{L^2((\cB')^\ext;\dC^2)}^2=\lambda(b,(\cB')^{\ext}).
\end{aligned}\end{equation}
Moreover, 
\begin{equation}\label{eq:traceu}\|u_\star|_{\p\Omg}\|^2_{L^2(\partial\Omg)}=
	\int_{\p\Omg} |\psi_\circ(0)|^2\dd \cH^1(x)=\|u_\circ|_{\p\cB'}\|_{L^2(\p\cB')}^2=1.
\end{equation}
Using again the co-area formula and the inequality~\eqref{eq:length_bnd} we get
\[
	\|u_\star\|^2_{L^2(\Omg^\ext)} = \int_0^\infty|\psi_\circ(t)|^2|\Sg_t| \dd t \le \int_0^\infty|\psi_\circ(t)|^2|\cC_t| \dd t = \|u_\circ\|^2_{L^2((\cB')^\ext)} <\infty.
\]
Thus, we get that $u_\star\in X_b(\Omg^{\ext})$ and using $u_\star$ as a trial state in \eqref{eq:def-ev-ext}, we conclude from~\eqref{eq:gradu} and~\eqref{eq:traceu} that
\[\lambda(b,\Omg^\ext)<\lambda(b,(\cB')^\ext).\qedhere\]
\end{proof}

\appendix

\section{Existence of ground states}
\subsection{Existence of a ground state for the magnetic Steklov problem on a bounded domain}\label{app:gs}

Foundations on the magnetic Steklov problem can be found in \cite[Appendix A.3]{CPS}. While \cite{CPS} addresses specifically the Aharonov-Bohm potential,  the analysis carries over to our setting of  a constant magnetic field.

For the sake of completeness, we provide here a proof that the infimum in \eqref{eq:def-ev} is attained at $u_\star\in H^1(\Omega)$ which satisfies \eqref{eq:DN-gs}.
For  $\beta\in\mathbb R$ and $b\geq 0$, we introduce the lowest eigenvalue of the magnetic Robin Laplacian (see \cite[Section~2]{KL22})
\[\mu^{\rm R}(\beta,b):=\inf_{\substack{u\in H^1(\Omega)\\ u\not=0}}\frac{\|(-\ii\nabla-b\bA)u\|^2_{L^2(\Omega;\dC^2)}+\beta\|u|_{\p\Omg}\|_{L^2(\partial\Omega)}^2}{\|u\|_{L^2(\Omega)}^2}.\]
The above infimum is attained on $u_\beta\in H^1(\Omega)\sm\{0\}$ which satisfies
\[
\begin{cases}
	(-\ii\nabla-b\bA)^2u_\beta=\mu^{\rm R}(\beta,b)u_\beta,&\quad \mbox{ in }\Omega,\\
 	\partial_\nu u_\beta+\beta u_\beta=0,&\quad\mbox{ on }\partial\Omega,
\end{cases} 
 \]
where $\nu$ is the outer unit normal to $\Omega$, and the above equation is understood in the following weak sense
\[\forall\,v\in H^1(\Omega),\qquad \langle (-\ii\nabla-b\bA)u_\beta,(-\ii\nabla -b\bA)v\rangle_{L^2(\Omega;\dC^2)}+\beta\langle u_\beta,v\rangle_{L^2(\partial\Omega)}=\mu^{\rm R}(\beta,b)\langle u_\beta,v\rangle_{L^2(\Omg)}.\]
By \cite[Proposition~5.3]{HKN}, there exists a unique $\beta_\star<0$ such that $\mu^{\rm R}(\beta_\star,b)=0$, and  
we have $\lambda(b,\Omega)=-\beta_\star$, where $\lambda(b,\Omega)$ is introduced in \eqref{eq:def-ev}.

Define  $u_\star=u_{\beta_\star}$. Then $u_\star$ satisfies \eqref{eq:DN-gs} and it is also a  minimizer of \eqref{eq:def-ev}.

\subsection{Existence of a ground state for the lowest fiber problem}\label{app:gs-1D}
Let $b, R > 0$ be fixed. We demonstrate here that the infimum
\[\lambda^{\rm 1D}(b,R)=\inf_{\substack{ f\in H^1((0,R),r\dd r)\\f(R)\not=0}}\frac{\displaystyle\int_0^R\left(|f'(r)|^2+\tfrac{b^2r^2}{4}|f(r)|^2\right)r\dd r}{R|f(R)|^2}\]
is attained on $f_\star\in H^1((0,R), r\dd r)$ with $f_\star(R)\ne 0$, and that $f_\star$ satisfies \eqref{eq:Steklov_System}. 

We introduce the lowest eigenvalue 
\[\mu^{\rm 1D}(\beta,b,R) :=\inf_{\substack{f\in H^1((0,R),r\dd r)\\f\not=0}}\frac{\displaystyle\int_0^R\Bigl(|f'(r)|^2+\tfrac{b^2r^2}{4}|f(r)|^2\Bigr)r\dd r + \beta R|f(R)|^2}{\displaystyle\int_0^R|f(r)|^2r\dd r},\]
of a Sturm--Liouville-type eigenvalue problem with Robin boundary conditions; here $\beta\in\dR$ is the boundary parameter.
The infimum is attained on the unique function (up to a constant factor) $0\ne f_\beta\in C^\infty([0,R])$ (see \cite[Proposition 3.2]{KL22}) and $f_\beta$ satisfies
\[\begin{cases}
    -f''_\beta(r) -\frac{1}{r}f_\beta'(r)+\frac{b^2r^2}{4}f_\beta(r)=\mu^{\rm 1D}(\beta,b,R)f_\beta(r),&\quad 0<r<R,\\[0.3ex]
    f'_\beta(R)=-\beta f_\beta(R).
\end{cases}\]
Moreover, by using constant function as the trial function we conclude that $\mu^{\rm 1D}(\beta,b,R)<0$ for $\beta<-R^3b^2/16$. {On the other hand, we have} $\mu^{\rm 1D}(0,b,R)>0$ (for $\beta=0$). By unique continuation property for ordinary differential equations we have $f_\beta(R) \ne 0$. Thus, we have that $\mu^{\rm 1D}(\beta,b,R)$ is strictly increasing in $\beta$.
Hence, there exists a unique $\beta_\star<0$ such that $\mu^{\rm 1D}(\beta_\star,b,R)=0$. By a direct comparison argument, we get that $\lambda^{\rm 1D}(b,R)=-\beta_\star$. 

We define $f_\star=f_{\beta_\star}$. Then $f_\star$ is a minimizer of the variational characterisation of $\lambda^{\rm 1D}(b,R)$ and it satisfies \eqref{eq:Steklov_System}.

\section{Variational characterisation of the lowest eigenvalue of the magnetic Dirichlet-to-Neumann map}

The lowest eigenvalue of the magnetic Dirichlet-to-Neumann map $\Lm_b$ (associated with the form $\frl_b$ in~\eqref{eq:frlb}) can be characterised by the min-max principle as
\begin{equation}\label{eq:lm1}
	\lm_1(\Lm_b) = 
	\inf_{\varphi\in H^{1/2}(\p\Omg)\sm\{0\}}
	\frac{\|(-\ii\nb - b{\bf A})\sfS_b \varphi\|^2_{L^2(\Omg;\dC^2)}}{\|\varphi\|_{L^2(\p\Omg)}^2}.
\end{equation}
\begin{prop}\label{prop:lowest}
Let $\lm(b,\Omg)$ be as in~\eqref{eq:def-ev}
and $\lm_1(\Lm_b)$ be as in~\eqref{eq:lm1}. Then $\lm(b,\Omg) = \lm_1(\Lm_b)$.	
\end{prop}
\begin{proof}
Using that $\sfS_b$ is a bijection from $H^{1/2}(\p\Omg)$ onto the solution space $\cN_b$ defined in~\eqref{eq:cN} we conclude that
\[
\lm_1(\Lm_b) 
= 
\inf_{u\in \cN_b\sm\{0\}}
\frac{\|(-\ii\nb - b{\bf A})u\|^2_{L^2(\Omg;\dC^2)}}
{\|u|_{\p\Omg}\|_{L^2(\p\Omg)}^2}.
\] 
In view of the inclusion $\cN_b\subset H^1(\Omg)$, we have $\lm_1(\Lm_b) \ge \lm(b,\Omg)$. In order to show the reverse inequality. We take any $u\in H^1(\Omg)$ and decompose it as $u = v_u + w_u$ where $v_u:=\sfS_b(u|_{\p\Omg})$ and $w_u:= u-v_u\in H^1_0(\Omg)$. Hence, we find using integration by parts that
\[
\begin{aligned}
\lm(b,\Omg) &=
\inf_{\begin{smallmatrix}u\in H^1(\Omg)\\ u|_{\p\Omg}\ne 0\end{smallmatrix}}
\frac{\|(-\ii\nb - b{\bf A})(v_u+w_u)\|^2_{L^2(\Omg;\dC^2)}}
{\|v_u|_{\p\Omg}\|_{L^2(\p\Omg)}^2}\\
&=
\inf_{\begin{smallmatrix} u\in H^1(\Omg)\\ u|_{\p\Omg}\ne 0\end{smallmatrix}}
\frac{\|(-\ii\nb - b{\bf A})v_u\|^2_{L^2(\Omg;\dC^2)}+
\|(-\ii\nb - b{\bf A})w_u\|^2_{L^2(\Omg;\dC^2)}}	{\|v_u|_{\p\Omg}\|_{L^2(\p\Omg)}^2}\\
&\ge
\inf_{\begin{smallmatrix}u\in H^1(\Omg)\\ u|_{\p\Omg}\ne0\end{smallmatrix}}
\frac{\|(-\ii\nb - b{\bf A})v_u\|^2_{L^2(\Omg;\dC^2)}}
{\|v_u|_{\p\Omg}\|_{L^2(\p\Omg)}^2}= \lm_1(\Lm_b).
\end{aligned}
\] 
Finally, we conclude that $\lm_1(\Lm_b) = \lm(b,\Omg)$.
\end{proof}

\subsection*{Acknowledgement} 
The first listed author (A.K.) acknowledges the support by a start-up fund at the American University of Beirut.

\end{document}

%% file: commands.tex
\newcommand\nb{\nabla}
\newcommand{\beq}{\begin{equation} \begin{split}}
\newcommand{\eeq}{\end{split} \end{equation}}
\newcommand\Sg{\Sigma}
\newcommand\Omg{\Omega}

\newcommand\ext{{\rm ext}}

\renewcommand\and{\qquad\text{and}\qquad}

\newcommand\sm{\setminus}

\newcommand\dl{\delta}

\newcommand\Lm{\Lambda}

\newcommand{\comm}[1]{}

\def\sfH{\mathsf{H}}\def\sfS{\mathsf{S}}
\def\sfS{\mathsf{S}}

\def\bm1{\mathbbm{1}}
\def\G{\Gamma}

\def\p{\partial}

\def\arr{\rightarrow}

\def\tt{\theta}

\def\lm{\lambda}

\def\ii{{\mathsf{i}}}
\def\p{\partial}

\def\kp{\kappa}

\def\sfH{\mathsf{H}}

\def\dd{{\,\mathrm{d}}}

\def\sfS{\mathsf{S}}

\def\sfU{\mathsf{U}}

\newcounter{counter_a}
\newenvironment{myenum}{\begin{list}{{\rm(\roman{counter_a})}}%
{\usecounter{counter_a}
\setlength{\itemsep}{1.ex}\setlength{\topsep}{0.8ex}
\setlength{\leftmargin}{5ex}\setlength{\labelwidth}{5ex}}}{\end{list}}

\usepackage[latin1]{inputenc}
\usepackage[T1]{fontenc}

\numberwithin{figure}{section}
\numberwithin{equation}{section}
\theoremstyle{plain}
\newtheorem*{thm*}{Theorem}
\newtheorem{thm}{Theorem}[section]

\newtheorem{lem}[thm]{Lemma}
\newtheorem{prop}[thm]{Proposition}

\newtheorem{cor}[thm]{Corollary}

\newtheorem{dfn}[thm]{Definition}
\theoremstyle{remark}
\newtheorem{remark}[thm]{Remark}

\theoremstyle{plain}

\newcommand{\beu}{\begin{equation*}}
\newcommand{\eeu}{\end{equation*}}
\newcommand{\besu}{\begin{equation*}
\begin{aligned}}
\newcommand{\eesu}{\end{aligned}
\end{equation*}}
\newcommand{\bes}{\begin{equation}
\begin{aligned}}
\newcommand{\ees}{\end{aligned}
\end{equation}}

\newcommand\cB{\mathcal B}

\newcommand\cF{\mathcal F}

\newcommand\cH{\mathcal H}

\newcommand\cN{\mathcal N}

\newcommand\frh{\mathfrak h}

\newcommand\ov{\overline}

\DeclareMathOperator\dom{dom}

\newcommand\void[1]{}

\def\ov{\overline}
\def\eps{\varepsilon}

\def\frl{{\mathfrak l}}

      \def\dC{{\mathbb C}}

   \def\dN{{\mathbb N}}   
      \def\dR{{\mathbb R}}

   \def\dZ{{\mathbb Z}}

   \def\sfH{{\mathsf H}}

\def\sfS{{\mathsf S}}      \def\sfU{{\mathsf U}}

   \def\cB{{\mathcal B}}   \def\cC{{\mathcal C}}
      \def\cF{{\mathcal F}}
   \def\cH{{\mathcal H}}   
      
   \def\cN{{\mathcal N}}

\def\N{\mathbb{N}}

%% file: KL_DN_final.bbl
\begin{thebibliography}{99}
%
\bibitem{AFK17}
P.~R.~S. Antunes, P.~Freitas, and D.~Krej\v{c}i\v{r}\'{\i}k. Bounds and
	extremal domains for {R}obin eigenvalues with negative boundary parameter,
{\it Adv. Calc. Var.} \textbf{10} (2017), 357--380.


\bibitem{BR12}
J.~Behrndt and J.~Rohleder. 
An inverse problem of Calder\'{o}n type with partial data,
{\it Commun. Partial Differ. Equations} {\bf 37} (2012),  1141--1159.


\bibitem{BR20}
J.~Behrndt and J.~Rohleder. 
Inverse problems with partial data for elliptic operators on unbounded Lipschitz domains,
{\it Inverse Probl.} {\bf 36} (2020),  035009.


\bibitem{B} H. Brezis. 
{\it Functional analysis, Sobolev spaces and partial differential equations.}
Universitext
Springer, New York, 2011.
%
\bibitem{B01}
F. Brock. An isoperimetric inequality for eigenvalues of the Stekloff problem, 
{\it ZAMM Z. Angew. Math. Mech.} {\bf 81} (2001), 69--71.
\bibitem{BFNT21}
D. Bucur, V. Ferone, C. Nitsch, and C. Trombetti. Weinstock inequality in higher
dimensions, {\it J. Differential Geom.} {\bf 118} (2021), 1--21.
%
\bibitem{B25}
L.~Bundrock. 
Optimizing the first Robin eigenvalue in exterior domains: the ball's local maximizing property,
{\it Ann. Mat. Pura Appl. (4)} {\bf 204} (2025),  1095--1117.
%
\bibitem{BGGLP25}
L.~Bundrock, A.~Girouard, D.~Grebenkov, M.~Levitin and I.~Polterovich. The exterior Steklov problem for Euclidean domains, \texttt{arXiv:2511.09490}.
%
\bibitem{CGH} T. Chakradhar, K. Gittins, G. Habib, N. Peyerimhoff.
A note on the magnetic Steklov operator on functions.
{\it Mathematika} {\bf 71} (2025), article no. e70037.
%
\bibitem{CPS} B. Colbois, L. Provenzano, and A. Savo. Isoperimetric Inequalities for the Magnetic Neumann and Steklov Problems with Aharonov-Bohm Magnetic Potential. {\it J. Geom. Anal.} {\bf 32} (2022), article no. 285.
%
\bibitem{CLPS} B.  Colbois,  C.  L\'ena,  L. Provenzano, and  A.  Savo. 
A reverse Faber-Krahn inequality for the magnetic Laplacian.  {\it J. Math. Pures Appl. (9)} {\bf 192} (2024),  Article ID 103632.
%
\bibitem{DKL} C. Dietze, A. Kachmar, and V. Lotoreichik. 
Isoperimetric inequalities for inner parallel curves. {\it J. Spectr. Theory} {\bf 14} (2024), 1537--1562.
%
\bibitem{EO}
A.\,F.\,M.~ter Elst and E.\,M.~Ouhabaz. 
The diamagnetic inequality for the Dirichlet-to-Neumann operator, {\it Bull. Lond. Math. Soc.} {\bf 54} (2022),
 1978--1997.
%
\bibitem{E96}
L.~Erd\H{o}s.
Rayleigh-type isoperimetric inequality with a homogeneous magnetic field,
\newblock {\em Calc. Var. Partial Diff. Equ.} {\bf 4} (1996),
283--292.
%
\bibitem{FK15}
P.~Freitas and D.~Krej\v{c}i\v{r}\'{\i}k. The first {R}obin eigenvalue
with negative boundary parameter, 
{\it Adv. Math.} \textbf{280} (2015), 322--339.
%
\bibitem{GP17}
A.~Girouard and I.~Polterovich. 
Spectral geometry of the Steklov problem, in: A.~Henrot (ed.), \emph{Shape optimization and spectral theory},  De Gruyter, Berlin, (2017), 120--148.
%
\bibitem{Ha} J.K. Hale.
{\it Ordinary differential equations.}
Wiley-Interscience, New York (1969).
%
\bibitem{HN-e} B. Helffer and F. Nicoleau. On the magnetic Dirichlet to Neumann operator on the exterior of the disk -- Diamagnetism, weak-magnetic field limit and flux effects. {\it J.  Math. Pures  Appl.}
{\bf 205}, article no. 103799.
%
\bibitem{HN} B. Helffer and F. Nicoleau. On the magnetic Dirichlet to Neumann operator on the disk -- strong diamagnetism and strong magnetic field limit--. {\it J. Geom. Anal.} {\bf 35} (2025), article no. 178.
%
\bibitem{HKN} B.  Helffer,  A.  Kachmar, and  F.  Nicoleau.  Asymptotics for the magnetic Dirichlet-to-Neumann eigenvalues in general domains.  arXiv:2501.00947  (2025).
%
\bibitem{HKN2} B.  Helffer,  A.  Kachmar, and  F.  Nicoleau. Flux effects on Magnetic Laplace and Steklov eigenvalues in the exterior of a disk. 
arXiv:2508.18119 (2025). 
%
\bibitem{Hu} A. Hurwitz. Sur quelques applications g\'eom\'etriques des s\'eries de Fourier. {\it Ann. Sci. Ecole
Norm. Sup. (3)} {\bf 19} (1902), 357--408.
%
\bibitem{KL22} A. Kachmar and V. Lotoreichik. On the isoperimetric inequality for the magnetic Robin Laplacian with negative boundary parameter. {\it J. Geom. Anal.} {\bf 32} (2022), article no. 182.
%
\bibitem{KL24}
A.~Kachmar and V.~Lotoreichik.
A geometric bound on the lowest magnetic Neumann eigenvalue via the torsion function, \textit{SIAM J. Math. Anal.} {\bf 56} (2024), 5723--5745.
%
\bibitem{KLS} A. Kachmar, V. Lotoreichik, M. Sundqvist.
On the Laplace operator with a weak magnetic field in exterior domains.
{\it Anal. Math. Phys.} {\bf 15} (2025), article no. 5.
%
\bibitem{Kato}
T.~Kato. \emph{Perturbation theory for linear operators}, Berlin, Springer-Verlag, 1995.
%
\bibitem{KP02}
S. Krantz and H. Parks. \emph{The implicit function theorem. History, theory, and applications},
Birkh\"{a}user, Boston, 2002.
%
\bibitem{KL18}
D.~Krej{\v{c}}i{\v{r}}{\'{\i}}k and V.~Lotoreichik.
Optimisation of the lowest {Robin} eigenvalue in the exterior of a
compact set,
{\it J. Convex Anal.} {\bf 25} (2018), 319--337.

\bibitem{KL20} D. Krej\v{c}i\v{r}\'ik and V. Lotoreichik. Optimisation of the lowest Robin
eigenvalue in the exterior of a compact set, II: non-convex domains
and higher dimensions, {\it Potential Anal.}  {\bf 52} (2020), 601--614.

\bibitem{KrL24}
D.~Krej\v{c}i\v{r}\'ik and V.~Lotoreichik.
Optimisation and monotonicity of the second Robin eigenvalue on a planar exterior domain,
{\it Calc. Var. Partial Differ. Equ.} {\bf 63} (2024), 223.

%
\bibitem{Lee}
J.~Lee. \emph{Introduction to smooth manifolds},
Springer, New York, 2013.
%
\bibitem{LMP}
M.~Levitin, D.~Mangoubi, and I.~Polterovich. 
\emph{Topics in spectral geometry}, American Mathematical Society, Providence, 2024.
%
\bibitem{M20}
B.\,S.~Mityagin. The zero set of a real analytic function,
\textit{Math. Notes} {\bf 107}  (2020), 529--530.
%
\bibitem{NIST}
F.~W.~J. Olver, D.~W. Lozier, R.~F. Boisvert, and C.~W. Clark.
{\em NIST Handbook of Mathematical Functions},
Cambridge University Press, Cambridge,
2010.
%
\bibitem{PW61}
L.\,E.~Payne and H.\,F.~Weinberger. 
Some isoperimetric inequalities for membrane frequencies and torsional rigidity, {\it 
J. Math. Anal. Appl.} {\bf 2} (1961),  210--216.
%
\bibitem{S01}
A.~Savo.
Lower bounds for the nodal length of eigenfunctions of the Laplacian, {\it 
Ann. Global Anal. Geom.} {\bf 19} (2001),  133--151.
%
\bibitem{SST}
K.~Shiohama, T.~Shioya, and M.~Tanaka. {\it
The geometry of total curvature on complete open surfaces.}
Cambridge University Press, Cambridge, 2003.
%
\bibitem{W54}	
R.~Weinstock. Inequalities for a classical eigenvalue problem, 
{\it J. Rational Mech. Anal.} {\bf 3} (1954), 745--753.
%
\end{thebibliography}
